\documentclass[12pt]{amsart}
\usepackage[T1]{fontenc}
\usepackage{amsmath,amssymb,amsthm,mathtools}
\usepackage{hyperref}

\newcommand{\R}{\mathbb R}
\newcommand{\C}{\mathbb C}
\newcommand{\Z}{\mathbb Z}
\newcommand{\D}{\mathbb D}
\newcommand{\CT}{T_{\C}}
\newcommand{\GL}{\mathrm{GL}}
\newcommand{\SL}{\mathrm{SL}}
\newcommand{\SO}{\mathrm{SO}}
\newcommand{\Spin}{\mathrm{Spin}}
\newcommand{\SU}{\mathrm{SU}}

\newcommand{\diag}{\operatorname{diag}}
\newcommand{\Int}{\operatorname{int}}
\newcommand{\en}{\mathfrak e}
\newcommand{\Real}{\operatorname{Re}}
\newcommand{\Imag}{\operatorname{Im}}
\newcommand{\Mon}{\operatorname{Mon}}

\parskip=\smallskipamount

\newtheorem{theorem}{Theorem}[section]
\newtheorem{lemma}[theorem]{Lemma}

\newtheorem{proposition}[theorem]{Proposition}

\theoremstyle{definition}
\newtheorem{remark}[theorem]{Remark}

\numberwithin{equation}{section}

\makeatletter
\@namedef{subjclassname@2020}{\textup{2020} Mathematics Subject Classification}
\makeatother

\begin{document}

\title[Hyperbolic complex tangencies on a four-sphere]{Hyperbolic complex tangencies\protect\\ on an embedded four-sphere}
\author[Marko Slapar]{Marko Slapar}
\address{Marko Slapar, Faculty of Mathematics and Physics\\ University of Ljubljana \\ Jadranska 19\\1000 Ljubljana, Slovenia\\ \newline
Faculty of Education\\ University of Ljubljana \\ Kardeljeva plo\v{s}\v{c}ad 16\\1000 Ljubljana, Slovenia\\
and\newline
  Institute of Mathematics, Physics and Mechanics\\Jadranska
  19\\1000 Ljubljana, Slovenia}
\email{marko.slapar@fmf.uni-lj.si}

\author[Sa\v{s}o Strle]{Sa\v{s}o Strle}
\address{Sa\v{s}o Strle, Faculty of Mathematics and Physics \\ University of Ljubljana \\ Jadranska 21 \\ 1000 Ljubljana, Slovenia }
\email{saso.strle@fmf.uni-lj.si}
\thanks{The first author was supported by the European Union through the ERC Advanced Grant HPDR, grant agreement No. 101053085, 
awarded to Franc Forstneri\v{c}, and by the research program P1-0291 and research project J1-70033 of ARIS, Republic of Slovenia.
The second author was supported by the research program P1-0288 of ARIS, Republic of Slovenia.}
\date{}

\begin{abstract}
We construct a smooth embedding of $S^4$ into $\mathbb C^4$ whose CR-singular locus is a connected surface of genus two consisting entirely of nondegenerate hyperbolic points. Its image is locally polynomially convex and admits a basis of smoothly bounded, strongly pseudoconvex Stein tubular neighbourhoods.
\end{abstract}
\maketitle
\pagestyle{plain}

\section{Introduction}

Let $f: M^n\to\C^n$ be a smooth immersion of an $n$-manifold $M$ into $\C^n$. The complex tangent space of $f(M)$ at a point $p\in M$ is defined as
$T_p^{\C}M=df_p(T_pM)\cap i\,df_p(T_pM).$
A point $p\in M$ is called \emph{totally real} if its complex tangent space is trivial, and \emph{CR-singular} otherwise. The \emph{order} of a CR-singular point is the complex dimension of its complex tangent space. If the immersion $f$ is generic, the set of CR-singular points of order $1$ forms a smooth codimension $2$ submanifold of $M$ \cite{Lai}. If $n\le 7$ and the immersion is generic, all CR-singular points have order $1$, and if $n\le 5$, we can further assume that CR-singular points are nondegenerate and can thus be classified as either elliptic, hyperbolic, or parabolic, according to their Bishop invariant \cite{Webster}.  

Bishop's classification \cite{Bishop} has direct implications for local holomorphic convexity. Elliptic points generate attached analytic discs and a nontrivial $(n+1)$-dimensional local hull of holomorphy, obstructing local polynomial convexity \cite{Huang}. At hyperbolic points, local polynomial convexity and uniform polynomial approximation hold under the additional flatness hypothesis \cite{Alagandala}. At parabolic points, quadratic data alone do not determine local convexity, and higher-order terms become essential, as already occurs for surfaces in $\C^2$ \cite{Joricke}.

CR-singularities have been extensively studied in dimension $2$ (see, for example, \cite{Forstneric}). In dimension $3$, Kasuya and Takase \cite{KT} studied the realization of CR-singular loci, while Elgindi \cite{Elgindi} investigated their Bishop invariants. In dimension $4$, relatively few explicit examples are known.

For a generic immersion $f:M^4\to\C^4$ of a closed oriented $4$-manifold, Webster \cite{Webster} established the identity
\[\chi(M)=\en(f)+e-h+p,\]
where $\en(f)$ denotes the normal Euler number, $e$ and $h$ are the Euler characteristics of the elliptic and hyperbolic loci, respectively, and $p$ is the parabolic index. In the same paper, he asked whether there exists an embedding of $S^4$ into $\C^4$ with no elliptic points. For a generic embedding satisfying the usual transversality conditions, the absence of elliptic points also excludes parabolic points. Since the normal Euler number of an embedding $S^4\hookrightarrow\C^4$ vanishes, the identity then implies that its CR-singular locus has Euler characteristic $-2$. If this locus is connected and orientable, it is therefore a surface of genus two.

In this paper, we answer Webster's question affirmatively by constructing a smooth embedding of $S^4$ into $\C^4$ with no elliptic points.
\begin{theorem}\label{thm:main}
There is a smooth embedding $f:S^4\hookrightarrow\C^4$, smoothly isotopic to the standard embedding in $\R^8$, such that the CR-singular locus is a connected surface of genus two and every CR-singular point is nondegenerate hyperbolic.
\end{theorem}

The constructed sphere also has the following convexity properties.

\begin{theorem}\label{thm:stein}
The sphere $f(S^4)\subset\C^4$ constructed in the proof of
Theorem~\ref{thm:main} is locally polynomially convex and admits
a basis of smoothly bounded, strongly pseudoconvex Stein tubular
neighbourhoods. \end{theorem}

For a real vector bundle $E$, write $E_{\C}=E\otimes_{\R}\C$, and in particular $\CT M=TM\otimes_{\R}\C$. For a four-manifold $M$ and a real-linear bundle map $A:TM\to M\times\C^4$, let
$A_{\C}:\CT M\to M\times\C^4$ denote its complex-linear extension. A real monomorphism $A$ is totally real precisely when $A_{\C}$ is an isomorphism.

\section{A tubular model and its formal extension}

Fix a smooth compact handlebody $H\subset\R^3$ with real-analytic boundary $\Sigma=\partial H$. We regard $H$ as $H\times\{0\} \subset \R^4=\R^3\times\R$. We identify a sufficiently small tubular neighbourhood of $\Sigma$ with $\Sigma\times\D_\rho$, where $\D_\rho \subset \C$ is the open disk of radius $\rho>0$ centered at the origin, using parametrization
\begin{equation}\label{eq:tube}
 \Psi(p,u+iv)=(p+u n(p),v),
\end{equation}
where $n$ denotes the outward unit normal vector field to $\Sigma$ in $\R^3$.

We start by embedding the model tubular neighbourhood $\Sigma\times\D_\rho$ into $\C^4$ by 
\begin{equation}\label{eq:q}
 F(p,z)=\left(p+\bar z^2n(p),z\right),
\end{equation}
where $\C^4 = \R^4 \otimes_{\R} \C$ is considered as the complexification of $\R^4$.
Since
$\Sigma$ is real analytic, near each $(p_0,0)$ there are local
holomorphic coordinates $(\eta_1,\eta_2,w,z)$, leaving the last
coordinate unchanged, in which the model is parametrized by
\begin{equation}\label{eq:normal-form}
  (x_1,x_2,z)\longmapsto(x_1,x_2,\bar z^{\,2},z),
  \qquad x_1,x_2\in\R,\quad z\in\C.
\end{equation}
Consequently, its CR-singular locus is $\Sigma\times\{0\}$, and all
its complex tangencies are nondegenerate hyperbolic points of Bishop
invariant $\infty$ \cite[\S1]{Webster}.

\begin{remark}
For the embedding construction in Theorem~\ref{thm:main}, the particular choice of quadratic perturbation $q(z,\bar z)=\bar z^2$ serves only to simplify the formulas. The construction works equally well with any fixed homogeneous quadratic polynomial
\[q(z,\bar z)=az^2+bz\bar z+c\bar z^2,\qquad |b|<2|c|.\]
So the Bishop invariant along $\Sigma$ can be any $\gamma=|c|/|b|\in (1/2,\infty].$ We retain the choice $q(z,\bar z)=\bar z^2$ for the Stein neighbourhood construction in Section~\ref{sec:stein}.
\end{remark}

\begin{lemma}\label{lem:embedding-extension}
After shrinking the tube, $F$ is the restriction of a smooth embedding $f_0:S^4\hookrightarrow\C^4$ which is smoothly isotopic to the standard embedding.
\end{lemma}

\begin{proof}
Start with a preliminary ambient isotopy that deforms the standard sphere so that the resulting sphere has a flat coordinate patch in
$\R^3\times\R\subset\C^3\times\C$ that contains $H\times\{0\}$.  In the coordinates~\eqref{eq:tube}, the embedding of the tubular neighbourhood is given by
\[F_0(p,u+iv)=(p+u n(p),v).\]
We first construct an isotopy of embeddings from $F_0$ to $F$. Write $N(p)=(n(p),0)$, $a=(0,1)$, and $b=(0,i)$ in
$\C^3\times\C$. These vectors form an orthonormal basis of the three-dimensional real subspace $V_p=\operatorname{span}_{\R}\{N(p),a,b\}$ of the normal space to $\Sigma\times\{0\}$ at $(p,0)$. Choose a smooth path of matrices in $\mathrm{SO}(3)$ from the identity to the cyclic permutation of the standard basis. For each $p$, let $R_s(p):V_p\to V_p$ be represented by this path in the ordered basis $(N(p),a,b)$. This gives a smooth family in $(p,s)$, with $R_0(p)=I$ and $
  R_1(p)N(p)=a$, $R_1(p)a=b$, $R_1(p)b=N(p)$.
Define
\[
  F_s(p,z)=(p,0)+R_s(p)\bigl(uN(p)+va\bigr)
                 +s\bigl(\bar z^2n(p),0\bigr).
\]
The endpoints are $F_0$ and $F$. Every displacement from $(p,0)$ lies in the ambient normal fibre of $\Sigma$ at $p$. On each such fibre, the map of the disk is an isometric linear inclusion plus a quadratic term whose $\mathcal C^1$ norm tends uniformly to zero with the radius. Consequently, for a sufficiently small common radius, $F_s$ is an isotopy of embeddings of the closed tube. The isotopy extension theorem~\cite{Hirsch} extends this isotopy to the ambient space and hence to the whole sphere.
\end{proof}

\begin{lemma}\label{prop:formal}
There exists a smooth real bundle monomorphism
\[
 A:TS^4\longrightarrow S^4\times\C^4
\]
which equals $dF$ on a neighbourhood of $\Sigma$ and is totally real on $S^4\setminus\Sigma$.
\end{lemma}

\begin{proof}
The complex rank-two bundle $\CT\Sigma$ is trivial. Choose a frame $e_1,e_2$ with
\begin{equation}\label{eq:frame-normalization}
 \det_{\C}(e_1,e_2,n)=1,
\end{equation}
where $n$ is the normal vector field and the three vectors span $C^3 \equiv C^3 \times \{0\}$. 
Using the product identification of the tube, extend
$e_1,e_2$ independently of $z$. Together with
$\partial_z,\partial_{\bar z}$, they give a complex frame
of $\CT S^4$ over the tube:
\begin{equation}\label{eq:adapted-frame}
 \mathcal E=(e_1,e_2,\partial_z,\partial_{\bar z}).
\end{equation}

The bundle $\CT S^4$ is trivial. By~\eqref{eq:frame-normalization}, the determinant of~\eqref{eq:adapted-frame} is constant along the surface. Hence a global complex frame $\mathcal G=(E_1,E_2,E_3,E_4)$ of $\CT S^4$ can be normalized so that
\[\det_{\C}(\mathcal G) =\det_{\C}(\mathcal E)\] 
over the tube. Its transition map from $\mathcal G$ to $\mathcal E$ therefore takes values in $\SL(4,\C)$. Write $\mathcal E=\mathcal G T$ on the tube $U$, where $T:U\to\SL(4,\C)$ is the transition map. Since $\SL(4,\C)\simeq\SU(4)$ is $2$-connected and the tube deformation retracts onto the two-dimensional surface $\Sigma$, this transition map is null-homotopic. Choose a smooth null-homotopy
\[H:U\times[0,1]\longrightarrow\SL(4,\C),\qquad H(x,0)=T(x),\quad H(x,1)=I.\]
Choose a smaller tube $U'\Subset U$ and a smooth function $\chi:S^4\to[0,1]$ that equals $0$ on $U'$ and $1$ on a neighbourhood of $S^4\setminus U$. Define
\[
\widetilde T(x)=
\begin{cases}
H(x,\chi(x)), & x\in U,\\
I, & x\notin U.
\end{cases}
\]
Since $\chi=1$ near the boundary of $U$, this defines a smooth map on $S^4$. Moreover, $\widetilde T=T$ on $U'$. Thus $\mathcal G\widetilde T$ is a global complex frame of $\CT S^4$ agreeing with $\mathcal E$ on $U'$. Denote this global frame again by $\mathcal E$, and shrink the tube so that it is contained in $U'$.

Let $S_p=dn_p\colon T_p\Sigma\to T_p\Sigma$, extended complex-linearly. For a tangent vector $\xi$ to $\Sigma$,
\begin{equation}
 d_{\C}F(\xi)=((I+\bar z^2S_p)\xi,0),\quad
 d_{\C}F(\partial_z)=(0,1),\quad
 d_{\C}F(\partial_{\bar z})=(2\bar zn(p),0).
 \label{eq:differential}
\end{equation}
In the frame $\mathcal E$ define, on the tube, the invertible matrix
\begin{equation}\label{eq:B}
 B(p,z)=
 \begin{bmatrix}
 (I+\bar z^2S_p)e_1 &(I+\bar z^2 S_p)e_2 &0&n(p)\\
 0&0&1&0
 \end{bmatrix}.
\end{equation}
Equations~\eqref{eq:differential} give
\begin{equation}\label{eq:factorization}
 d_{\C}F=B(p,z)\diag(1,1,1,2\bar z).
\end{equation}
At $z=0$, the normalization~\eqref{eq:frame-normalization} gives $\det_{\C}B(p,0)=-1$. As for $T$ above it follows that $B$ is null-homotopic in $\GL(4,\C)$ over the tube. Using the same collar construction as above, deform $B$ to a constant invertible matrix (of negative determinant) across a collar surrounding a smaller tube, keeping $B$ unchanged on that tube. Extending by that constant matrix beyond the collar gives a smooth map
\[B:S^4\longrightarrow\GL(4,\C)\]
agreeing with the original $B$ on the smaller tube.

It remains to extend the scalar factor $2\bar z$ without adding zeros. Choose a smooth function $r:\R^3\to\R$ which is negative on $\Int H$, positive outside $H$, and equals the signed distance $u$ near $\Sigma$. Let $t$ be the fourth real coordinate. Choose a compactly supported smooth function $\eta:\R^3\times\R\to[0,1]$ equal to $1$ on a neighbourhood of $H\times\{0\}$ and of the tube. Put
\begin{equation}\label{eq:global-h}
 h(x,t)=2\left(1-\eta(x,t)+\eta(x,t)r(x)\right)-2i\eta(x,t)t.
\end{equation}
Outside a compact set, $h=2$, so it extends smoothly over the rest of the sphere. If $h=0$, its imaginary part implies either $\eta=0$ or $t=0$. The first possibility is excluded by its real part. When $t=0$ and $x\in H$, we have $\eta=1$, so $h=0$ exactly on $\partial H=\Sigma\times\{0\}$; when $x\notin H$, its real part is strictly positive. Hence
\[
 h^{-1}(0)=\Sigma,\qquad
 h\circ\Psi=2u-2iv=2\bar z
 \quad\text{near }\Sigma.
\]

Using the global frame $\mathcal E$, define
\begin{equation}\label{eq:formal-A}
 A_{\C}=B\diag(1,1,1,h),\qquad A=A_{\C}|_{TS^4}.
\end{equation}
Near $\Sigma$, this is $dF$ by~\eqref{eq:factorization}. On the complement of $\Sigma$, $A_{\C}$ is an isomorphism, and hence $A$ is a totally real monomorphism. On $\Sigma$, it is a real monomorphism because it equals the differential of the embedding $F$.
\end{proof}

\section{The relative obstruction and the normal Euler number}

Give $S^4$ its standard orientation and $\C^4$ its complex orientation. For a real bundle monomorphism $A:TS^4\to S^4\times\R^8$, equip its normal bundle
\[
 \nu_A=(S^4\times\R^8)/A(TS^4)
\]
with the induced orientation, and let
\[
 \en(A)=\langle e(\nu_A),[S^4]\rangle
\]
be the normal Euler number of $A$. For an immersion $g$, put $\en(g)=\en(dg)$.

\begin{proposition}\label{lem:relative-euler}
Let $U\subset S^4$ be a compact tubular neighbourhood of a closed surface, and let $A_0,A_1:TS^4\to S^4\times\R^8$ be smooth real bundle monomorphisms which agree on a neighbourhood of $U$. They are homotopic through real bundle monomorphisms, fixed on a neighbourhood of $U$, if and only if
\[
 \en(A_0)=\en(A_1).
\]
\end{proposition}

\begin{proof}
The necessity follows from homotopy invariance of the Euler class. 

Recall that an orientable $\R^4$-bundle over $S^4$ is determined by the homotopy class of its clutching map, so an element of $\pi_3(\SO(4))$. The fibration $\SO(4)\hookrightarrow\SO(5)\to S^4$, where the projection is given by $A \mapsto A e_5$, is the principal frame bundle of the tangent bundle of $S^4$. Since $\pi_4(\SO(5))$ is torsion, this yields a short exact sequence
\begin{equation}\label{eq:4-bundles}
 0\longrightarrow\pi_4(S^4)
 \longrightarrow\pi_3(\SO(4))
 \longrightarrow\pi_3(\SO(5))\longrightarrow 0.
\end{equation}
The group $\Spin(4)\cong S^3 \times S^3$ is the double cover of $\SO(4)$, hence $\pi_3(\SO(4))\cong\Z\oplus\Z$, whereas $\pi_3(\SO(5))\cong\Z$ is already in the stable range. The above sequence is then
\[
0 \to \Z \to Z  \oplus \Z \to \Z \to 0,
\]
where the stabilization map sends $(a, b)$ to $a+b$ and the Euler number of a bundle corresponding to $(a,b)$ is $a-b$. Since $TS^4$ is stably trivial and its Euler number is $2$, its clutching function generates the kernel of the stabilization map; indeed, the tangent bundle corresponds to the identity map of $S^4$. 

Choose a slightly larger compact tubular neighbourhood $K$ of the surface such that $U\subset\Int K$ and $A_0=A_1$ on $K$. Consider the bundle
\[\mathcal M=\Mon(TS^4,S^4\times\R^8)\longrightarrow S^4.\]
Its fibre $F=\Mon(\R^4,\R^8)$ deformation retracts onto the Stiefel manifold $V_{8,4}$. Hence $F$ is $3$-connected and the fibration $\SO(4)\hookrightarrow\SO(8)\to V_{8,4}$ yields 
\begin{equation}\label{eq:stiefel-sequence}
 0\longrightarrow\pi_4(V_{8,4})
 \longrightarrow\pi_3(\SO(4))
 \longrightarrow\pi_3(\SO(8))\longrightarrow0.
\end{equation}
From the stability $\pi_3(\SO(5))=\pi_3(\SO(8))$ it follows that this sequence is isomorphic to \eqref{eq:4-bundles} and in particular 
$\pi_4(F)\cong\Z$. Using this identification, an element of $\pi_4(F)$ maps to its complementary four-plane bundle whose Euler number is twice that element.

Since $F$ is 3-connected and the sphere is 4-dimensional, by relative obstruction theory for homotopies of sections \cite[\S3.3, pp.~101--102]{Hatcher}, the only obstruction to a homotopy between $A_0$ and $A_1$ fixed on $K$ is
\[d_K(A_0,A_1)\in H^4(S^4,K;\pi_4(F))\cong H^4(S^4,K;\Z).\]
The coefficient system is constant because $S^4$ is simply connected. 
Since the regular neighbourhood $K$ deformation retracts onto the surface, the obstruction actually lives in $H^4(S^4;\Z)$. By the above, twice the evaluation of this obstruction on the fundamental class of the sphere is equal to the difference of the Euler classes, $\en(A_1)-\en(A_0)$, which vanishes by assumption. Thus $d_K(A_0,A_1)=0$, and the required homotopy exists relative to $K$. By relative smooth approximation, the homotopy may be chosen smooth, while keeping its endpoints and a neighbourhood of $U$ fixed.
\end{proof}

We use the following two relative forms of the totally real $h$-principle. They follow from Gromov's results~\cite[\S2.4.5]{Gromov}; see also~\cite[\S\S19.3--19.4]{EM} and the discussion in~\cite[p.~355]{Forstneric}.

\begin{enumerate}
\item[(i)] Let $W$ be a compact smooth four-manifold with boundary, let $f:W\to\C^4$ be smooth, and let $A:TW\to W\times\C^4$ be a totally real monomorphism equal to $df$ on a boundary collar. There is a totally real immersion $g:W\to\C^4$, equal to $f$ on a smaller collar, such that $dg$ is homotopic to $A$ through totally real monomorphisms relative to that collar.
\item[(ii)] If $f:W\hookrightarrow\C^4$ is an embedding, totally real near the boundary, and $df$ is homotopic through real monomorphisms, relative to a boundary collar, to a totally real monomorphism, then $f$ admits an arbitrarily $\mathcal C^0$-small isotopy to a totally real embedding, fixed on a smaller collar.
\end{enumerate}

The immersion statement starts with a formal solution and does not require a homotopy from $df$ to $A$. The embedding statement requires precisely such a homotopy.

\begin{lemma}\label{prop:normal-euler}
For the formal monomorphism $A$ of Lemma~\ref{prop:formal},
\begin{equation}\label{eq:normal-euler}
 \en(A)=\chi(S^4)+\chi(\Sigma)=2+\chi(\Sigma).
\end{equation}
\end{lemma}

\begin{proof}
Choose nested closed tubes $U_0\subset\Int U_1$ around $\Sigma$ on which $A=dF$, and an embedding $f_0$ as in Lemma~\ref{lem:embedding-extension}, equal to $F$ on a neighbourhood of $U_1$. On
\[W=S^4\setminus\Int U_0,\]
$A$ is totally real and equals $df_0$ on a boundary collar. Apply (i), preserving a collar contained in $U_1\setminus\Int U_0$, and glue the resulting immersion to $F$ on $U_0$. This gives a smooth immersion $g:S^4\to\C^4$ whose CR singular locus is $\Sigma$, all tangencies are hyperbolic, and $dg$ is homotopic to $A$ relative to a neighbourhood of $\Sigma$. Hence $\en(g)=\en(A)$.

For a compact real $4$-manifold immersed in $\C^4$, Webster's formula is
\[\chi(M)=\en(g)+e-h+p,\]
where $e=\chi(N_e)$, $h=\chi(N_h)$, and $p$ is the parabolic index~\cite[Theorem~5.1]{Webster}. In our case $N_e=N_p=\varnothing$, and $N_h=\Sigma$. The formula reduces to
\[
 \chi(S^4)=\en(g)-\chi(\Sigma),
\]
which proves~\eqref{eq:normal-euler}.
\end{proof}

\section{Proof of Theorem~\ref{thm:main}}

\begin{proof}[Proof of Theorem~\ref{thm:main}]
Take $H$ to have genus two, so that $\chi(\Sigma)=-2$. Construct $F$, $f_0$, and $A$ as in Section~2. Lemma~\ref{prop:normal-euler} gives
\[\en(A)=2-2=0.\]
An oriented closed four-manifold embedded in $\R^8$ has normal Euler number equal to its self-intersection number, which is zero in $\R^8$. In particular, $\en(df_0)=0$.

Choose $0<\rho_0<\rho_1<\rho_2<\rho$ and set
\[U_j=\Psi\bigl(\Sigma\times\overline{\D}_{\rho_j}\bigr),\qquad j=0,1,2.\]
Then
\[U_0\subset\Int U_1\subset U_1\subset\Int U_2,\]
and $A=df_0=dF$ on a neighbourhood of $U_2$.
Applying Proposition~\ref{lem:relative-euler} relative to $U_2$ gives a homotopy
\[A_s:TS^4\longrightarrow S^4\times\C^4,\qquad A_0=df_0,\quad A_1=A,\]
through real bundle monomorphisms, fixed on $U_2$ and hence on a neighbourhood of $U_1$.

Set $W=S^4\setminus\Int U_0$. The endpoint $A_1$ is totally real on $W$, and the homotopy is fixed on the collar $U_2\setminus\Int U_0$, where $f_0=F$ is already totally real. Apply the relative embedding statement~(ii), choosing the smaller fixed collar to contain $U_1\setminus\Int U_0$. This gives an arbitrarily $\mathcal C^0$-small isotopy of $f_0|_W$, fixed on $U_1\setminus\Int U_0$, whose endpoint is totally real.

This isotopy can be glued to the unchanged embedding on $U_0$ without introducing intersections. Indeed, the compact sets
\[f_0(U_0)\quad\text{and}\quad f_0(S^4\setminus\Int U_1)\]
are disjoint and therefore have positive distance. Choose the above isotopy uniformly smaller than half this distance. Its moving part then stays disjoint from $f_0(U_0)$, while the fixed collar ensures smooth gluing. We obtain an isotopy of the entire embedded sphere.

The endpoint $f$ is totally real away from $\Sigma$ and agrees with~\eqref{eq:q} near $\Sigma$. The initial embedding $f_0$ was isotopic to the standard sphere by Lemma~\ref{lem:embedding-extension}, so the same holds for $f$.
\end{proof}

\section{Stein neighbourhoods}\label{sec:stein}

\begin{proof}[Proof of Theorem~\ref{thm:stein}] Put $M=f(S^4)$ and identify its CR-singular locus with
$\Sigma_0=\Sigma\times\{0\}$. Near $\Sigma_0$, $M$ agrees with the
model $F(p,z)=(p+\bar z^{\,2}n(p),z)$.

Let $\pi$ and $r$ denote the normal projection onto $\Sigma$ and the signed distance to $\Sigma$ on a tubular neighbourhood of $\Sigma\subset\R^3$. Since $\Sigma$ is real analytic, these maps have holomorphic extensions $P\colon\mathcal U\to\C^3$ and $W\colon\mathcal U\to\C$ on a common neighbourhood $\mathcal U$ of $\Sigma$ in $\C^3$. They satisfy
\[P(p+wn(p))=p,\qquad W(p+wn(p))=w, \qquad p\in\Sigma,\quad |w|\text{ small}.\]
At $p\in\Sigma$, the differentials of $P$ and $W$ are the tangential and normal projections, respectively. Thus, after shrinking, $(P,W)$ is a holomorphic immersion. Moreover, $P(P(X))=P(X)$ by
holomorphic continuation, so $P(X)\in\R^3$ implies $p=P(X)\in\Sigma$.
The displayed identities and local injectivity of $(P,W)$ then give
$X=p+W(X)n(p)$. Consequently, near $\Sigma_0$, $M$ is given precisely by the equations
\[\Imag P(X)=0,\qquad W(X)=\bar z^{\,2}.\]
After shrinking, choose an ambient neighbourhood $\mathcal V_0\subset\mathcal U\times\D_1$ of $\Sigma_0$ on which the above equations describe $M$. Define
\begin{equation}\label{eq:Phi} \Phi(X,z)=(1+|z|^2)
       \bigl(|\Imag P(X)|^2+|W(X)-\bar z^{\,2}|^2\bigr).
\end{equation}
The function $\Phi$ is nonnegative on $\mathcal V_0$ and satisfies
\[
  \Phi^{-1}(0)=M\cap\mathcal V_0.
\]
To check its Levi form, consider the function
\[\Psi(\zeta,w,z)=(1+|z|^2)\bigl(|\Imag\zeta|^2+|w-\bar z^{\,2}|^2\bigr), \qquad (\zeta,w,z)\in\C^3\times\C\times\C.\]
We have $\Phi(X,z)=\Psi(P(X),W(X),z)$. The Levi matrix of $\Psi$ is
\begin{equation}\label{eq:aux-Levi-matrix}
  \begin{pmatrix}
    \dfrac{1+|z|^2}{2}I_3 & 0 & -iz\,\Imag\zeta\\[5pt]
    0 & 1+|z|^2 & z(\bar w-z^2)\\[3pt]
    i\bar z(\Imag\zeta)^T & \bar z(w-\bar z^{\,2}) & \Psi_{z\bar z}
  \end{pmatrix},
\end{equation}
where direct differentiation gives
\begin{equation}\label{eq:aux-last-entry}
  \Psi_{z\bar z}={}|\Imag\zeta|^2+|w-\bar z^{\,2}|^2+4|z|^2+8|z|^4-4\Real(wz^2).
\end{equation}
The upper-left $4\times4$ block is positive definite. Its Schur
complement is
\begin{equation}\label{eq:aux-Schur}
\begin{aligned}
  &\Psi_{z\bar z}
   -\frac{|z|^2}{1+|z|^2}
       \bigl(2|\Imag\zeta|^2+|w-\bar z^{\,2}|^2\bigr)\\
  &\quad=\frac{1-|z|^2}{1+|z|^2}|\Imag\zeta|^2
       +\frac{|w-(3+2|z|^2)\bar z^{\,2}|^2}{1+|z|^2}
       +4|z|^2(1-|z|^4).
\end{aligned}
\end{equation}
For $|z|<1$, all terms on the right are nonnegative, and their sum
vanishes precisely on
\[
  E=\{\Imag\zeta=0,\ w=z=0\}.
\]
Thus $\Psi$ is plurisubharmonic for $|z|<1$ and strictly
plurisubharmonic outside $E$, in particular outside its zero set.
Since $(X,z)\mapsto(P(X),W(X),z)$ is a holomorphic immersion whose
inverse image of $E$ is $\Sigma_0$, holomorphic pullback shows that
$\Phi$ is plurisubharmonic and strictly plurisubharmonic outside
$\Sigma_0$.

To check the real normal Hessian, set
\[
  a=\bigl(\Imag P_1,\Imag P_2,\Imag P_3,
           \Real(W-\bar z^{\,2}),\Imag(W-\bar z^{\,2})\bigr).
\]
At $(p,0)\in\Sigma_0$, the differential of $\Imag P$ has rank two on the imaginary tangential directions, while the real and imaginary parts of $W$ provide two additional independent normal differentials. Thus $da$ has rank four along $\Sigma_0$. By continuity and compactness, its rank is at least four near $\Sigma_0$; on $M\cap \mathcal V_0$ it is at most four because $a=0$ there. Hence $\operatorname{rank}da=4$ and $\ker da=TM$ along $M\cap \mathcal V_0$, after shrinking the neighbourhood. Since $\Phi=(1+|z|^2)|a|^2$ and $a=0$ on $M\cap \mathcal V_0$, we have
\[
  \operatorname{Hess}\Phi
    =2(1+|z|^2)\sum_{j=1}^5 da_j\otimes da_j
    \qquad\text{on }M\cap \mathcal V_0.
\]
Consequently, at every point $q\in M\cap\mathcal V_0$, the real Hessian of $\Phi$ is positive semidefinite of rank four, with kernel exactly $T_qM$. In particular, it is positive definite in the normal directions, and $d\Phi=0$ on $M\cap \mathcal V_0$.

We now extend the construction to a neighbourhood of the whole sphere. The squared Euclidean distance
\[d_M(q)=\operatorname{dist}(q,M)^2\]
is smooth on a tubular neighbourhood of $M$. Both $d_M$ and its differential vanish on $M$, and the real Hessian of $d_M$ is positive definite in the normal directions along $M$. It is strictly plurisubharmonic near every totally real point of $M$. Choose a smooth cutoff function $\chi\in C_c^\infty(\mathcal V_0)$, with $0\le\chi\le1$, which equals one near $\Sigma_0$, and extend it by zero to $\C^4$. On a sufficiently small ambient neighbourhood $\mathcal V$ of $M$, put
\[
  \rho=\chi\Phi+(1-\chi)d_M,
\]
with $\chi\Phi$ extended by zero outside $\mathcal V_0$. Where the two functions are glued, both they and their differentials vanish on $M$. All terms involving derivatives of $\chi$ therefore disappear from their real and complex Hessians along $M$. In particular,
\[\mathcal L_\rho(q;\cdot)=\chi(q)\mathcal L_\Phi(q;\cdot)+(1-\chi(q))\mathcal L_{d_M}(q;\cdot),\qquad q\in M,\]
where the first summand is understood as zero outside the support of $\chi$. The transition occurs on a compact subset of $M\setminus\Sigma_0$, where both Levi forms are positive definite. By continuity and compactness, positivity persists on an ambient neighbourhood of this transition set. After shrinking around $M$, the function $\rho$ is therefore plurisubharmonic and strictly plurisubharmonic off $\Sigma_0$. Moreover, $\rho\geq0$, $\rho^{-1}(0)=M$, $d\rho|_M=0$, and its real normal Hessian is positive definite along $M$.

Since $\rho^{-1}(0)=M$, $d\rho=0$ on $M$, and the real Hessian of $\rho$ is positive definite in the normal directions, after shrinking $\mathcal V$ around $M$, the sufficiently small sublevel sets
\[\Omega_\varepsilon=\{q\in\mathcal V:\rho(q)<\varepsilon\}
\]
form a basis of relatively compact tubular neighbourhoods of $M$. Their boundaries are regular level sets of $\rho$ and are strongly pseudoconvex, since $\rho$ is strictly plurisubharmonic outside $\Sigma_0$. Hence each $\Omega_\varepsilon$ is Stein.

Finally, fix $q\in M$ and choose $s>0$ such that
$\overline B(q,s)\subset\mathcal V$. For $K=M\cap\overline B(q,s)$, polynomial convexity of the closed ball gives $\widehat K\subset\overline B(q,s)\subset \mathcal V$. By the plurisubharmonic characterization of polynomial hulls,
\[\rho\leq\sup_K\rho=0\quad\text{on }\widehat K.\]
Since $\rho\geq0$ and $\rho^{-1}(0)=M$, it follows that $\widehat K\subset M\cap\overline B(q,s)=K$. Hence $K$ is polynomially convex, proving local polynomial convexity of $M$.
\end{proof}

\subsection*{AI Declaration} The initial proof of Lemma 2.2 was provided by ChatGPT 5.6 Sol (OpenAI), which suggested the isotopy construction extending the local tubular model to an embedding of the sphere. ChatGPT was also used to assist with mathematical calculations, to explore proof arguments, and to improve the exposition. The authors have carefully checked and revised the AI-assisted arguments and take full responsibility for the mathematical correctness and content of the paper.

\bibliographystyle{amsplain}

\end{document}